\documentclass[11pt,reqno]{amsart}
\usepackage{mathrsfs}

\usepackage{hyperref}

\usepackage{graphicx}

\usepackage{xcolor}

\newcommand{\bpartial}{\bar{\partial}}
\newcommand{\p}{\partial}

\newcommand{\ol}{\overline}
\newcommand{\ul}{\underline}

\newtheorem{theorem}{Theorem}[section]
\newtheorem{lemma}[theorem]{Lemma}

\newtheorem{corollary}[theorem]{Corollary}

 \theoremstyle{definition}
\newtheorem{definition}[theorem]{Definition}

\theoremstyle{remark}

\numberwithin{equation}{section}

\begin{document}
\setlength{\baselineskip}{1.2\baselineskip}

\title[The parabolic flows for complex quotient equations]
{The parabolic flows for complex quotient equations}
\thanks{The author is supported by National Natural Science Foundation of China (No. 11501119).}

\author{Wei Sun}

\address{Department of Mathematics, Shanghai University, Shanghai 200444, China}
\email{sunweimath@shu.edu.cn}

\begin{abstract}
We apply the parabolic flow method to solving complex quotient equations on closed K\"ahler manifolds. We study the parabolic equation and prove the convergence. As a result, we solve the complex quotient equations.
\end{abstract}

\maketitle

\section{Introduction}
\label{gjf-int}

Let $(M,\omega)$ be a closed K\"ahler manifold of complex dimension $n \geq 2$, and $\chi$ a smooth closed real $(1,1)$ form in $\Gamma^k_\omega$, where $\Gamma^k_\omega$ is the set of all the real $(1, 1)$ forms whose eigenvalue sets with respect to $\omega$ belong to $k$-positive cone in $\mathbb{R}^n$.  In any local coordinate chart, we write
\begin{equation*}
	\chi = \frac{\sqrt{-1}}{2} \sum_{i,j} \chi_{i\bar j} d z^i \wedge d \bar z^j \qquad \text{and} \qquad \omega = \frac{\sqrt{-1}}{2} \sum_{i,j} g_{i\bar j} d z^i \wedge d \bar z^j .
\end{equation*}
In this paper, we study the following form of parabolic equations, for $ n \geq k > l \geq 0$
\begin{equation}
\label{pfcqe-int-flow-equation}
	\frac{\p u}{\p t} = \log \frac{\chi^k_u \wedge \omega^{n - k}}{\chi^l_u \wedge \omega^{n - l}} - \log \psi,
\end{equation}
with initial condition $u(x,0) = 0$, where $\psi \in C^{\infty} (M)$ is positive and $\chi_u $ is the abbreviation for $\chi + \frac{\sqrt{- 1}}{2} \p\bpartial u $.

The study of the parabolic flows is motivated by complex equations
\begin{equation}
\label{elliptic-equation}
	\chi^k_u \wedge \omega^{n - k} = \psi \chi^l_u \wedge \omega^{n - l}, \qquad \chi_u \in \Gamma^k_\omega .
\end{equation}
When $\psi$ is constant, it must be $c$ defined by
\begin{equation}
\label{gjf-int-condition}
	 c := \frac{\int_M \chi^k \wedge \omega^{n - k}}{\int_M \chi^l \wedge \omega^{n - l}}.
\end{equation}
These equations include some important geometric equations, which have attracted much attention in mathematics and physics since the breakthrough of Yau\cite{Yau78} (see also Aubin\cite{Aubin78}) in Calabi conjecture\cite{Calabi56}. 
The most famous examples are probably the complex Monge-Amp\`ere equation and Donaldson equation\cite{Donaldson99a}, which respectively correspond to
\begin{equation}
	\chi^n_u  = \psi \omega^{n }, \qquad \chi_u \in \Gamma^n_\omega .
\end{equation}
and 
\begin{equation}
	\chi^n_u  = \frac{\int_M \chi^n }{\int_M \chi \wedge \omega^{n - 1}} \chi_u \wedge \omega^{n - 1}, \qquad \chi_u \in \Gamma^n_\omega .
\end{equation} 
Since equation~\eqref{elliptic-equation} is fully nonlinear elliptic, a classical way to solve it is the continuity method. Using this method, the complex Monge-Amp\`ere equation was solved by Yau\cite{Yau78}, while Donaldson equation was independently solved by Li, Shi and Yao\cite{LiShiYao15}, Collins and Sz\`ekelyhidi\cite{CollinsSzekelyhidi14} and the author\cite{Sun2013e}. There have been many extensive studies for equation~\eqref{elliptic-equation} on closed complex manifolds, for example, \cite{Cherrier87,DK,Hou09,HouMaWu10,TWWvY2014,TWv10a,TWv10b,ZhangDK15,Zhang10}.

Besides the continuity method, the parabolic flow method also have the potential to solve complex quotient equations. The result of Yau\cite{Yau78} was reproduced by Cao\cite{Cao85} through K\"ahler-Ricci flow. Donaldson equation  was actually first solved  via the $J$-flow by Song and Weinkove\cite{SW08}. 
There are many results regarding different flows on closed complex manifolds, and we refer readers to \cite{CollinsSzekelyhidi14,FLM11,Gill11,LejmiSzekelyhidi15,Sun2013p,TWv11,Weinkove04,Weinkove06}.
In this paper, we preliminarily explore the parabolic flow equation~\eqref{pfcqe-int-flow-equation} for complex quotient equations,  and reprove the solvability of the corresponding complex quotient equations (see \cite{Szekelyhidi2014b,Sun2014e}).

To solve equation~\eqref{pfcqe-int-flow-equation}, some extra condition is in need. We assume that there is a real-valued $C^2$ function $\underline u$ satisfying $\chi_{\underline u} \in \Gamma^k_\omega$ and 
\begin{equation*}
	k \chi_{\underline u}^{k - 1} \wedge \omega^{n - k} > l \psi\chi_{\underline u}^{l - 1} \wedge \omega^{n - l} .
\end{equation*} 
For convenience, we adopt an equivalent definition of $\underline u$ due to  Sz\'ekelyhidi\cite{Szekelyhidi2014b},  which is called $\mathcal{C}$-subsolution.
\begin{definition}
\label{definition-subsolution}
We say that a $C^2$ function $\ul u$ is a $\mathcal{C}$-subsolution to \eqref{pfcqe-int-flow-equation} if $\chi_{\ul u} \in \Gamma^k_\omega$,  and at each point $x \in M$, the set
\begin{equation}
	\Big\{\tilde \chi \in \Gamma^k_\omega \,\Big|\,  \tilde \chi^k \wedge \omega^{n - k} \leq \psi \tilde \chi^l \wedge \omega^{n - l} \text{ and } \tilde \chi - \chi_{\ul u} \geq 0 \Big\}
\end{equation}
is bounded.
\end{definition}
When $l = 0$, a natural $\mathcal{C}$-subsolution is $\ul u \equiv  0$. In other cases, the existence of a $\mathcal{C}$-subsolution is not easy to verify in applications. In the case $\psi \equiv  c$, Sz\'ekelyhidi\cite{Szekelyhidi2014b} propose a conjecture that  the existence of a $\mathcal{C}$-subsolution is equivalent to that for all $m$ dimensional subvariety $V \subset M$, where $m = n - l, \cdots, n - 1$, 
\begin{equation}
\label{conjecture-criterion}
	\frac{k!}{(m - n + k)!} \int_V \chi^{m - n + k} \wedge \omega^{n - k} -  \frac{l! c}{(m - n + l)!} \int_V \chi^{m - n + l} \wedge \omega^{n - l} > 0 .
\end{equation}
For Donaldson equation,  the conjecture was verified in dimensional $2$ by Lejmi and Sz\`ekelyhidi\cite{LejmiSzekelyhidi15} and on toric manifolds by Collins and Sz\`ekelyhidi\cite{CollinsSzekelyhidi14}. 
In view of these results, we expect the conjecture to hold in general cases.

The following is the main result of the paper.
\begin{theorem}
\label{main-theorem}
Let $(M^n,\omega)$ be a closed K\"ahler manifold of complex dimension $n$ and $\chi$ a smooth closed real $(1,1)$ form in $\Gamma^k_\omega$. Suppose that there is a $\mathcal{C}$-subsolution $\ul u$ and  $\psi \geq c$ for all $x \in M$, where $c$ is an invariant as defined in  \eqref{gjf-int-condition}.
Then there exists a long time solution $u$ to equation~\eqref{pfcqe-int-flow-equation}.
Moreover, the normalization $\hat u$ of $u$ is $C^\infty$ convergent to a smooth function $\hat u_\infty$ where $\hat u$ is defined later in Section~\ref{preliminary}.
Consequently, there is a unique real number $b$ such that the pair $(\hat u_\infty , b)$ solves 
\begin{equation}
	\frac{\chi^k_u \wedge \omega^{n - k}}{\chi^l_u \wedge \omega^{n - l}} = e^b \psi .
\end{equation}
\end{theorem}
Very recently, Phong and T\^{o}~\cite{PhongTo17} were also able to prove the result. 

To obtain the $L^\infty$ estimate, we adapt an approach due to Blocki\cite{Blocki2005a} and Sz\'ekelyhidi\cite{Szekelyhidi2014b}. However the Alexandroff-Bakelman-Pucci maximum principle\cite{Tso85} is dependent on time $t$, which implies that the $L^\infty$ estimate for $u$ probably blows up as time $t$ approaches $\infty$. In this case, higher order estimates were also dependent on time, which means we were unable to prove the convergence. To eliminate the influence of time, we apply Alexandroff-Bakelman-Pucci maximum principle locally in time $t$. 
The proof of the second order estimates follows the work of Hou, Ma and Wu\cite{HouMaWu10}.  In order to apply the work, we need to improve a key lemma in \cite{Sun2014e}. 
The gradient estimate thus follows from the blow up argument of Dinew and Kolodziej\cite{DK} and Gill\cite{Gill14}.
The techniques in this papers can be applied to other flows, while equation~\eqref{pfcqe-int-flow-equation} is the appropriate form to deal with real number $b$ in Theorem~\ref{main-theorem}.

\section{Preliminary}
\label{preliminary}

In this section, we shall state some notations and prove some preliminary results.

We follow the notations in \cite{GSun12, Sun2014e,Sun2013p}. Let $S_k (\boldsymbol{\lambda})$ denote the $k$-th elementary symmetric polynomial of $\boldsymbol{\lambda} \in \mathbb{R}^n$,
\begin{equation}
	S_k (\boldsymbol{\lambda}) = \sum_{1 \leq i_1 < \cdots < i_k \leq n} \lambda_{i_1} \cdots \lambda_{i_k} \,.
\end{equation}
For $\{i_1, \cdots, i_s\} \subseteq\{1,\cdots,n\}$,
\begin{equation}
	S_{k;i_1\cdots i_s} (\boldsymbol{\lambda}) = S_k (\boldsymbol{\lambda}|_{\lambda_{i_1} = \cdots = \lambda_{i_s} = 0}) .
\end{equation}
By convention, $S_0 (\boldsymbol{\lambda}) = 0$ and thus $S_{-1;i_1\cdots i_s} (\boldsymbol{\lambda}) = S_{-2;i_1\cdots i_s} (\boldsymbol{\lambda}) = 0$.
We express $X := \chi_u$ and hence in any local coordinate chart,
\begin{equation}
	X_{i\bar j} = X \Big(\frac{\p}{\p z^i} , \frac{\p}{\p \bar z^j}\Big) = \chi_{i\bar j} + u_{i\bar j} .
\end{equation}
When no confusion occurs in a fixed local coordinate chart, we also abuse $X$ and $\chi$ et al. to denote the corresponding Hermitian matrices. We define $\boldsymbol{\lambda}_*(X)$ as the eigenvalue set of $X$ with respect to $\omega$ and thus write $S_k (X) = S_k (\boldsymbol{\lambda}_* (X))$. For simplicity, we use $S_k$ to denote $S_k (X)$. In any local coordinate chart, equation~\eqref{pfcqe-int-flow-equation} can be written as
\begin{equation}
\label{pfcqe-int-flow-equation-1}
	\frac{\p u}{\p t} = \log \frac{S_k}{S_l} - \log \Psi , 
\end{equation}
with $\Psi = \frac{C^k_n}{C^l_n}\psi$ and $C^k_n = \frac{n!}{(n - k)! k!}$.
For convenience, we use indices to denote covariant derivatives with respect the Chern connection $\nabla$ of $\omega$. 

Fixing a point $p$, we can choose a normal chart around $p$ such that $g_{i\bar j} = \delta_{ij}$ and $X$ is diagonal. Then there is a natural ordered eigenvalue set
\begin{equation}
	\boldsymbol{\lambda}_* (X) = (X_{1\bar 1} , \cdots , X_{n\bar n}) ,
\end{equation}
and we write
\begin{equation}
	S_{k;i_1 \cdots i_s} = S_{k; i_1 \cdots i_s} (\boldsymbol{\lambda}_* (X)) .
\end{equation}
Moreover, we have the following equalities
\begin{equation}
	X_{i\bar jj} - X_{j\bar ji} = \chi_{i\bar jj} - \chi_{j\bar ji} ,
\end{equation}
and 
\begin{equation}
	X_{i\bar ij\bar j} - X_{j\bar ji\bar i} = R_{j\bar ji\bar i} X_{i\bar i} - R_{i\bar ij\bar j} X_{j\bar j} - G_{i\bar ij\bar j}
\end{equation}
where
\begin{equation}
	G_{i\bar ij\bar j} = \chi_{j\bar ji\bar i} - \chi_{i\bar ij\bar j} + \sum_m R_{j\bar ji\bar m} \chi_{m \bar i} - \sum_m R_{i\bar ij\bar m} \chi_{m\bar j} .
\end{equation}
For simplicity, we define
\begin{equation}
	F(\chi_u) := \log \frac{S_k (\chi_u)}{S_l (\chi_u)}
\end{equation}
and
\begin{equation}
	F^{i\bar j} = \frac{\p F}{\p u_{i\bar j}}.
\end{equation}
Then at $p$ under the chosen chart, $\{F^{i\bar j}\}$ is diagonal and 
\begin{equation}
	F^{i\bar i} = \frac{S_{k - 1;i}}{S_K} - \frac{S_{l - 1;i}}{S_l} .
\end{equation}

Differentiating equation \eqref{pfcqe-int-flow-equation-1}  at  $p$ under such a normal coordinate chart,
\begin{equation}
\label{gjf-int-derivative-time}
\p_t (\p_t u) = \sum_i F^{i\bar i} (\p_t u)_{i\bar i} ,
\end{equation}
\begin{equation}
\label{gjf-int-derivative-1}
	\p_t u_m = \sum_i F^{i\bar i} X_{i\bar im} - \p_m \log \Psi ,
\end{equation}
and
\begin{equation}
\begin{aligned}
	\p_t u_{m\bar m} \leq&\, \sum_i F^{i\bar i} X_{i\bar im\bar m} - \sum_{i\neq j} \Big(\frac{S_{k-2;ij}}{S_k} - \frac{S_{l - 2;ij}}{S_l}\Big) X_{i\bar jm} X_{j\bar i\bar m} \\
	&\,- \bpartial_m\p_m \log \Psi .
\end{aligned}
\end{equation}
Applying the maximum principle to equation~\eqref{gjf-int-derivative-time}, we see that $\p_t u$ reaches its extremal values at $t = 0$. 
Thanks to the boundedness of $\p_t u$, it is easy to see that the flow remains in $\Gamma^k_\omega$ at any time.

Fixing a $\mathcal{C}$-subsolution $\ul u$, the set
\begin{equation}
	\Big\{\tilde \chi \in \Gamma^k_\omega (x) \,\Big|\,  \tilde \chi^k \wedge \omega^{n - k} \leq e^\lambda \psi \tilde \chi^l \wedge \omega^{n - l} \text{ and } \tilde \chi - \chi_{\ul u}(x) \geq - \lambda g_{i\bar j} \Big\}
\end{equation}
is also bounded if $\lambda > 0$ is small enough.
Without loss of generality, we may assume that $\sup_M \ul u = - 2 \lambda$ in this paper.  To obtain the second order estimate, we improve a key lemma in \cite{Sun2014e} to fit in with the parabolic case.
\begin{lemma}
\label{quotient-alternate-lemma}
Under the assumptions of Theorem~\ref{main-theorem}, there is a constant $\theta > 0$ such that we have
either
\begin{equation}
\label{quotient-alternate-lemma-1}
	\sum_i F^{i\bar i} (u_{i\bar i} - \ul u_{i\bar i}) \leq F(\chi_u) - \log \Psi - \theta \Big( 1+ \sum_i F^{i\bar i} \Big) ,
\end{equation}
or
\begin{equation}
\label{quotient-alternate-lemma-2}
	F^{j\bar j} \geq \theta \Big(1  +  \sum_i F^{i\bar i}  \Big), \qquad \forall j = 1, \cdots, n.
\end{equation}
\end{lemma}
\begin{proof}
Without loss of generality, we may assume that $X_{1\bar 1} \geq \cdots \geq X_{n\bar n}$. Thus 
\begin{equation}
F^{n\bar n} \geq \cdots \geq F^{1\bar 1} .
\end{equation}

If $\lambda > 0$ is small enough, $\chi - \lambda \omega$ and $\ul u$ still satisfy Definition~\ref{definition-subsolution}. Since $M$ is compact, there are uniform constants $N > 0$ and $\sigma > 0$ such that
\begin{equation}
	F (\chi') >  \log {\Psi} + \sigma ,
\end{equation}
where
\begin{equation}
	\chi' =\chi_u - \lambda g +  
	\begin{Bmatrix}
	& N &	  &	 			& \\
	&	 &	0 &   			& \\
	&	 &	  &	\ddots	 & \\
	&	 &	  &				 & 0
	\end{Bmatrix}_{n\times n} .
\end{equation}
Direct calculation shows that
\begin{equation}
\begin{aligned}
	\sum_i F^{i\bar i} (u_{i\bar i} - \ul u_{i\bar i}) 
	=&\, \sum_i F^{i\bar i} X_{i\bar i} - \sum_i F^{i\bar i} \chi'_{i\bar i} + N F^{1\bar 1} - \lambda \sum_i F^{i\bar i} \\
	\leq&\, F (\chi_u) - F (\chi') + N F^{1\bar 1} - \lambda \sum_i F^{i\bar i} \\
	\leq&\, F (\chi_u) - \log \Psi - \sigma - \lambda \sum_i F^{i\bar i} + N F^{1\bar 1} .
\end{aligned}	
\end{equation}
If 
\begin{equation}
	 \frac{ \min \{\sigma, \lambda\} }{2} \Big( 1 + \sum_i F^{i\bar i} \Big)\geq  N F^{1\bar 1} ,
\end{equation}
we obtain \eqref{quotient-alternate-lemma-1}; otherwise, inequality~\eqref{quotient-alternate-lemma-2} has to be true.

\end{proof}

The argument can be applied to elliptic and parabolic equations with $\mathcal{C}$-subsolution. For the parabolic equation, we shall use the following corollary.
The argument applies to general functions $F$.
\begin{corollary}
\label{quotient-alternate-corollary}
Under the assumption of Lemma~\ref{quotient-alternate-lemma} and additionaly assuming that $X_{1\bar 1} \geq \cdots \geq X_{n\bar n}$, there is a constant $ \theta > 0 $ such that we have either
\begin{equation}
\label{quotient-alternate-corollary-1}
\begin{aligned}
	 \sum_i F^{i\bar i} (u_{i\bar i} - \ul u_{i\bar i}) - \p_t u \leq&\,   - \theta \Big(1 +  \sum_i F^{i\bar i} \Big),
\end{aligned}	
\end{equation}
or
\begin{equation}
\label{quotient-alternate-corollary-2}
	F^{1\bar 1} X_{1\bar 1} \geq \theta\Big(1 +  \sum_i F^{i\bar i} \Big) .
\end{equation}
\end{corollary}
\begin{proof}

Note that
\begin{equation}
	X_{1\bar 1} \geq \frac{S_1}{n} > 0,
\end{equation}
and
\begin{equation}
	F \Big(\frac{S_1}{n} g\Big) \geq  F( X) = \p_t u + \log \Psi
\end{equation}
 by concavity and symmetry of $F$. The latter inequality implies that there is a constant $\sigma>0$ such that $S_1 > \sigma$, and hence the former shows that $X_{1\bar 1} > n \sigma$. Combining Lemma~\ref{quotient-alternate-lemma} and the fact that $X_{1\bar 1} > n \sigma$, we finish the proof.

\end{proof}


We adapt the general $J$-functionals\cite{Chen00b,FLM11}.  
Let $\mathcal{H}$ be the space
\begin{equation}
	\mathcal{H} := \{ u\in C^\infty(M) \;|\; \chi_u \in \Gamma^k_\omega\} .
\end{equation}
For any curve $v(s) \in \mathcal{H}$, we define the funtional $J_l$ by
\begin{equation}
\label{uniform-J-functional-defition-derivative}
	\frac{d J_l}{d s} = \int_M \frac{\partial v}{\partial s} \chi^l_v \wedge \omega^{n - l} .
\end{equation}
Then we have a formula for $J_l(u)$,
\begin{equation}
\label{uniform-J-functional-defition-formula}
J_l (u) = \int^1_0 \int_M \frac{\partial v}{\partial s} \chi^l_v \wedge \omega^{n - l} ds,
\end{equation}
where $v(s)$ is an arbitrary path in $\mathcal{H}$ connecting $0$ and $u$. 
Thanks to the closedness of $\chi$ and $\omega$, those functionals are independent from choices of the path.
If the integration is over $v(s) = s u$, 
\begin{equation}
\label{uniform-J-functional-calculation-line}
\begin{aligned}
	J_l (u) 
	&= \int^1_0 \int_M u \left(s \chi_u + (1 - s)\chi\right)^l \wedge \omega^{n - l} ds \\
	&= \frac{1}{l + 1} \sum^l_{i = 0}  \int_M u \chi^i_u \wedge \chi^{l - i} \wedge \omega^{n - l} .
\end{aligned}
\end{equation}
Along the solution flow $u(x,t)$ to equation~\eqref{pfcqe-int-flow-equation}, we have
\begin{equation}
\label{uniform-theorem-kahler-decreasing-J-functional}
\begin{aligned}
	\frac{d}{dt} J_l (u) 
	&= \int_M \left(\log \frac{\chi^k_u \wedge \omega^{n - k}}{\chi^l_u \wedge \omega^{n - l}} - \log \psi\right) \chi^l_u \wedge \omega^{n - l} \\
	&\leq \log c\int_M \chi^l_u \wedge \omega^{n - l}  - \int_M \log \psi \chi^l_u \wedge \omega^{n - l} \\
	&\leq 0 .
\end{aligned}
\end{equation}
Computing $J_l$ on an arbitrary function flow $u(x,t)$ starting from $0$ to $T$, 
it follows that
\begin{equation}
\label{uniform-J-functional-calculation-flow}
\begin{aligned}
	J_l(u(T)) 
	&= \int^T_0 \frac{d J_l}{d t} dt .
\end{aligned}
\end{equation}
For the solution flow $u(x,t)$ to equation~\eqref{pfcqe-int-flow-equation}, let 
\begin{equation}
\label{uniform-time-normalized-solution-defintion}
\hat u = u - \frac{J_l (u)}{\int_M \chi^l\wedge \omega^{n - l}} .
\end{equation}
By \eqref{uniform-theorem-kahler-decreasing-J-functional}, we know that $\p_t \hat u \geq \p_t u$.

\section{The $L^\infty$ estimate}
\label{uniform}

In this section, we shall prove the $L^\infty$ estimate. We follow the approach of Sz\'ekelyhidi\cite{Szekelyhidi2014b} based on the method of Blocki\cite{Blocki2005a}.
\begin{theorem}
\label{uniform-theorem}
Let $u \in C^2 (M \times [0,T))$ be an admissible solution to equation~\eqref{pfcqe-int-flow-equation}. Suppose that there is a $\mathcal{C}$-subsolution $\ul u$ and  $\psi \geq c$ for all $x \in M$.
Then there exists a uniform constant $C > 0$ such that for any $t\in [0,T)$
\begin{equation}
	\sup_M u (x,t) - \inf_M u(x,t) = \sup_M \hat u(x,t) - \inf_M \hat u(x,t) < C .
\end{equation}
\end{theorem}

To prove Theorem~\ref{uniform-theorem}, we need the following lemma. The argument follows closely those in \cite{Weinkove06,Sun2013p}, so we omit the proof. 
\begin{lemma}
\label{uniform-lemma-1}
\begin{equation}
\label{uniform-lemma-1-inequality}
	0 \leq \sup_M \hat u(x , t) \leq - C_1 \inf_M \hat u(x , t) + C_2 \quad\text{ and } \quad\inf_M \hat u(x,t) \leq 0 .
\end{equation}
\end{lemma}
This lemma tells us that it suffices to find a lower bound for $\inf_M (\hat u - \ul u)(x , t)$.

\begin{proof}[Proof of Theorem~\ref{uniform-theorem}]
We claim that
\begin{equation}
\label{uniform-lower-bound}
	\inf_M (\hat u - \ul u) (x,t) > - 2 \sup_{M\times\{0\}} |\p_t u| - C_0 ,
\end{equation}
where $C_0 \geq 0$ is to be determined later.

Since $\p_t u$ reaches its extremal values at $t = 0$, 
\begin{equation}
\label{gjf-int-max-3}
	|\p_t \hat u| \leq 2 \sup_{M\times\{0\}} |\p_t u| .
\end{equation}
So when $t \leq 1$, we have
\begin{equation}
	\inf_M (\hat u - \ul u) (x,t) \geq \inf_M \hat u (x,t) + 2 \lambda \geq - 2 t \sup_{M\times\{0\}} |\p_t u| + 2 \lambda .
\end{equation}
Therefore if the lower bound~\eqref{uniform-lower-bound} does not hold, there must be time $t_0 > 1$ such that
\begin{equation}
	\inf_M (\hat u - \ul u) (x , t_0) = \inf_{M \times [0,t_0]} (\hat u - \ul u) (x,t) = - 2 \sup_{M\times\{0\}} |\p_t u| - C_0 .
\end{equation}
We may assume that $(\hat u - \ul u) (x_0 ,  t_0) = \inf_M (\hat u - \ul u) (x , t_0) < 0$. Following the approach of Sz\'ekelyhidi\cite{Szekelyhidi2014b}, we work in local coordinates around $x_0$, where $x_0$ is the origin and the coordinates are defined for $B_1 = \{z : |z| < 1\}$. Let $v = \hat u - \ul u - \epsilon + \epsilon |z|^2 - \epsilon (t - t_0) - \inf_M (\hat u - \ul u) (x , t_0)$ for some small $\epsilon > 0$. We may assume that $\epsilon < \lambda$. It is easy to see that when $t = t_0 - 1$
\begin{equation}
	v = \hat u - \ul u + \epsilon |z|^2  - \inf_M (\hat u - \ul u) (x, t_0) \geq 0,
\end{equation}
and when $|z|^2 = 1$, $t \leq t_0$
\begin{equation}
	v = \hat u - \ul u - \epsilon (t - t_0) - \inf_M (\hat u - \ul u) (x , t_0) \geq 0.
\end{equation}
Moreover,
\begin{equation}
	\inf_{M \times [t_0 - 1,t_0]} v = \inf_{M\times \{t_0\}} v = v (x_0,t_0) =  - \lambda .
\end{equation}
Define the set for $- v$ on $ [t_0 - 1, t_0]$,
\begin{equation}
\begin{aligned}
	\Phi (y, t) = \Big\{(p,h) \in \mathbb{R}^{2 n + 1}\,\Big| &- v (x , s) \leq - v(y , t) + p \cdot (x - y), \\
	& h =-  v (y,t) - p \cdot y, \forall x \in B_1, s\in [t_0 - 1, t] \Big\} .
\end{aligned}
\end{equation}
Then we define the contact set 
\begin{equation}
	\Gamma_{- v} = \Big\{ (y , t) \in B_1 \times [t_0 - 1, t_0] \,\Big|\, \Phi (y,t) \neq \emptyset\Big\} .
\end{equation}
In $\Gamma_{- v}$, it must be true that $D^2_x v \geq 0$ and $\p_t v \leq 0$. Therefore,
\begin{equation}
\{u_{i\bar j}\} - \{\ul u_{i \bar j}\} \geq - \epsilon I,
\end{equation}
and 
\begin{equation}
	\frac{\chi^k_u \wedge \omega^{n - k}}{\chi^l_u \wedge \omega^{n - l}} \leq e^{\epsilon }\psi .
\end{equation}
If $\epsilon$ is chosen small enough,  we obtain an bound $|u_{i\bar j}| < C$ in $\Gamma_{-v}$.
By Alexandroff-Bakelman-Pucci maximum principle\cite{Tso85} for parabolic equations, we have 
\begin{equation}
\begin{aligned}
	\epsilon &\leq  C \left[\int_{ \Gamma_{- v} \cap \{v < 0\}} - \p_t v \det  (D^2_x v) dx dt \right]^{\frac{1}{2n + 1}} \\
	&\leq  C \left[\int_{ \Gamma_{- v} \cap \{v < 0\}} - \p_t v 2^{2 n} (\det (v_{i\bar j}))^2 dx dt \right]^{\frac{1}{2n + 1}} .
\end{aligned}
\end{equation}
Because of the boundedness of $u_{i\bar j}$ and $\p_t u$, it follows that 
\begin{equation}
	\epsilon \leq C \left|\Gamma_{-v} \cap \{v < 0\} \right|^\frac{1}{2n + 1} .
\end{equation}
When $v < 0$, 
\begin{equation}
	\hat u  < \ul u + \epsilon - \epsilon |z|^2 + \epsilon  (t - t_0) + \inf_M (\hat u - \ul u) (x , t_0) 
	< \inf_M (\hat u - \ul u) (x , t_0) .
\end{equation}
So
\begin{equation}
\begin{aligned}
	\epsilon^{2n + 1} &\leq  C \left|M \times [t_0 - 1 , t_0] \cap \Big\{\hat u  <  \inf_M (\hat u - \ul u) (x , t_0)\Big\} \right|  \\
	&\leq C \int^{t_0}_{t_0 - 1} \frac{||\hat u^- (x,t)||_{L^1}}{|\inf_M (\hat u - \ul u) (x , t_0)|} dt \\
	&\leq C \int^{t_0}_{t_0 - 1} \frac{||\hat u(x,t) - \sup_M \hat u(x,t)||_{L^1}}{|\inf_M (\hat u - \ul u) (x , t_0)|} dt
\end{aligned}
\end{equation}
Using the Green's function of $\omega$, we have a uniform bound for 
$||\hat u(x,t) - \sup_M \hat u(x,t)||_{L^1}$. Therefore, there is a uniform constant $C_0 \geq 0$ such that 
\begin{equation}
	\inf_{M \times [0,t_0]} (\hat u - \ul u) (x,t) > - C_0 ,
\end{equation}
which contradicts the definition of $t_0$.

\end{proof}

\section{The second order estimate}
\label{second}

In this section, we shall prove the second order estimates.

\begin{theorem}
\label{second-lemma-estimate}
There exists a constant $C$ depending on  $\sup_{M\times [0,T)} |\hat u | $ such that for any $t' \in [0,T)$, 
\begin{equation}
\label{second-lemma-inequality}
	\sup_M |\p\bpartial u| \leq C \Bigg(\sup_{M\times [0,t']} |\nabla u|^2 + 1\Bigg),
\end{equation}
at any time $t \in [0,t']$.
\end{theorem}
\begin{proof}
Following the work of Hou, Ma and Wu\cite{HouMaWu10}, we define
\begin{equation}
	H (x ,\xi) = \log \Big(\sum_{i,j} X_{i\bar j} \xi^i \bar \xi^j\Big) + \varphi (|\nabla u|^2) + \rho (\hat u - \ul u)
\end{equation}
where
\begin{equation}
\begin{aligned}
	\varphi (s) &= - \frac{1}{2} \log \Big(1 - \frac{s}{2 K}\Big), \qquad \text{for } 0 \leq s \leq K - 1, \\
	\rho (t) &= - A \log \Big(1 + \frac{t}{2L}\Big), \qquad \text{for } - L + 1 \leq t \leq L - 1,
\end{aligned}
\end{equation}
with
\begin{equation*}
\begin{aligned}
	K &:= \sup_{M\times [0,t']} |\nabla u|^2 + \sup |\nabla \ul u|^2 + 1, \\
	L &:= \sup_{M\times [0,T)} |\hat u| + \sup_M | \ul u| + 1, \\
	A &:= 3 L (C_0 + 1)
\end{aligned}
\end{equation*}
and $C_0$ is to be specified later. Note that 
\begin{equation}
\label{second-test-function-varphi-1}
	\frac{1}{2K} \geq \varphi' \geq \frac{1}{4 K} > 0, \;\; \varphi'' = 2 (\varphi')^2 > 0
\end{equation}
and
\begin{equation}
\label{second-test-function-rho-1}
	\frac{A}{L} \geq - \rho' \geq \frac{A}{3 L} = C_0 + 1,\;\; \rho'' \geq \frac{2\epsilon}{1 - \epsilon} (\rho')^2,\;\; \text{ for all } \epsilon \leq \frac{1}{2 A + 1}.
\end{equation}
The function $H$ must achieves its maximum at some point $(p,t_0)$ in some unit direction of $\eta$. Around $p$, we choose a normal chart such that $X_{1\bar 1} \geq \cdots \geq X_{n\bar n}$, and $X_{1\bar 1} = X_{\eta\bar \eta}$ at $p$. 

Define
\begin{equation}
	H_t := \frac{\p_t u_{1\bar 1}}{X_{1\bar 1}} + \varphi' \p_t (|\nabla u|^2) + \rho' \p_t (\hat u - \ul u) ,
\end{equation}
\begin{equation}
	H_i := \frac{X_{1\bar 1 i}}{X_{1\bar 1}} + \varphi' \p_i (|\nabla u|^2) + \rho' (u_i - \ul u_i),
\end{equation}
and
\begin{equation}
\begin{aligned}
	H_{i\bar i} :=&\, \frac{X_{1\bar 1i\bar i}}{X_{1\bar 1}} - \frac{|X_{1\bar 1i}|^2}{X^2_{1\bar 1}} + \varphi'' |\p_i (|\nabla u|^2)|^2 + \varphi' \bpartial_i\p_i (|\nabla u|^2) \\
	&\, + \rho'' |u_i - \ul u_i|^2 + \rho' ( u_{i\bar i} - \ul u_{i\bar i}) .
\end{aligned}
\end{equation}
At $(p,t_0)$, we have $H_t \geq 0$, $H_i = 0$ and $H_{i\bar i} \leq 0$. Thus 
\begin{equation}
\label{second-proof-derivative-time}
	\frac{\p_t u_{1\bar 1}}{X_{1\bar 1}}  \geq - \varphi' \p_t (|\nabla u|^2) - \rho' \p_t u,
\end{equation}
\begin{equation}
\label{second-proof-derivative-1}
	\frac{X_{1\bar 1i}}{X_{1\bar 1}} = - \varphi' \p_i (|\nabla u|^2) - \rho' (u_i - \ul u_i) ,
\end{equation}
and
\begin{equation}
\label{second-proof-derivative-2}
\begin{aligned}
	0 \geq&\, \frac{1}{X_{1\bar 1}} (X_{i\bar i1\bar 1} + R_{i\bar i1\bar 1} X_{1\bar 1} - R_{1\bar 1i\bar i} X_{i\bar i} - G_{1\bar 1i\bar i}) - \frac{|X_{1\bar 1i}|^2}{X^2_{1\bar 1}} \\
	&\, + \varphi'' |\p_i (|\nabla u|^2)|^2 + \varphi' \bpartial_i\p_i (|\nabla u|^2) + \rho'' |u_i - \ul u_i|^2 + \rho' (u_{i\bar i} - \ul u_{i\bar i}) .
\end{aligned}
\end{equation}
Multiplying \eqref{second-proof-derivative-2} by $F^{i\bar i}$ and summing it over index $i$,
\begin{equation}
\label{second-proof-control-1}
\begin{aligned}
	0 \geq&\, \frac{1}{X_{1\bar 1}} \p_t u_{1\bar 1} + \frac{1}{X_{1\bar 1}} \bpartial_1\p_1 (\log \Psi) + \frac{1}{X_{1\bar 1}} \sum_{i \neq j} \Big(\frac{S_{k - 2;ij}}{S_k} - \frac{S_{l - 2;ij}}{S_l} \Big)X_{i\bar j1} X_{j\bar i\bar 1} \\
	&\, + \frac{(k - l) \inf_p R_{p\bar p 1\bar 1}}{X_{1\bar 1}}  + \inf_p R_{p\bar p 1\bar 1} \sum_i F^{i\bar i} - \frac{\sup_p  G_{1\bar 1p\bar p}}{X_{1\bar 1}} \sum_i F^{i\bar i} \\
	&\, - \frac{1}{X^2_{1\bar 1}} \sum_i F^{i\bar i} |X_{1\bar 1i}|^2  + \varphi'' \sum_i F^{i\bar i} |\p_i (|\nabla u|^2)|^2 + \varphi' \sum_i F^{i\bar i} \bpartial_i\p_i (|\nabla u|^2)  \\
	&\, + \rho'' F^{i\bar i} |u_i - \ul u_i|^2 + \rho' \sum_i F^{i\bar i} (u_{i\bar i} - \ul u_{i\bar i}).
\end{aligned}
\end{equation}
Substituting \eqref{second-proof-derivative-time} into \eqref{second-proof-control-1},
\begin{equation}
\label{second-proof-control-1-1}
\begin{aligned}
	0 \geq&\,\frac{1}{X_{1\bar 1}} \bpartial_1\p_1 (\log \Psi) + \frac{1}{X_{1\bar 1}} \sum_{i \neq j} \Big(\frac{S_{k - 2;ij}}{S_k} - \frac{S_{l - 2;ij}}{S_l} \Big)X_{i\bar j1} X_{j\bar i\bar 1} \\
	&\, + \frac{(k - l) \inf_p R_{p\bar p 1\bar 1}}{X_{1\bar 1}}  + \inf_p R_{p\bar p 1\bar 1} \sum_i F^{i\bar i }  - \frac{\sup_p  G_{1\bar 1p\bar p}}{X_{1\bar 1}} \sum_i F^{i\bar i}  \\
	&\, - \frac{1}{X^2_{1\bar 1}} \sum_i F^{i\bar i}  |X_{1\bar 1i}|^2 + \varphi'' \sum_i F^{i\bar i} |\p_i (|\nabla u|^2)|^2 + \varphi' \sum_i F^{i\bar i} \bpartial_i\p_i (|\nabla u|^2)  \\
	&\,    + \rho'' \sum_i F^{i\bar i} |u_i - \ul u_i|^2  + \rho' \sum_i F^{i\bar i} (u_{i\bar i} - \ul u_{i\bar i}) - \varphi' \p_t (|\nabla u|^2)  - \rho' \p_t u .
\end{aligned}
\end{equation}
Direct calculation shows that,
\begin{equation}
	\p_t (|\nabla u|^2) 
	= 2 \sum_j \mathfrak{Re} \{\p_t u_j u_{\bar j}\},
\end{equation}
\begin{equation}
	\p_i (|\nabla u|^2) = \sum_j (- u_j \chi_{i\bar  j} + u_{j i} u_{\bar j}) + u_i X_{i\bar i} ,
\end{equation}
and
\begin{equation}
\begin{aligned}
	\bpartial_i\p_i (|\nabla u|^2) 
	\geq&\, - 2 \sum_j \mathfrak{Re} \{\chi_{i\bar ij} u_{\bar j}\} + \sum_{j,k} R_{i\bar ij\bar k} u_k u_{\bar j} + \frac{1}{2} X^2_{i\bar i} \\
	&\, - 2 \chi^2_{i\bar i} + 2 \sum_j \mathfrak{Re}\{ X_{i\bar ij} u_{\bar j}\} .
\end{aligned}
\end{equation}
We control some terms in \eqref{second-proof-control-1},
\begin{equation}
\label{second-proof-control-2}
\begin{aligned}
	&\, \varphi' \sum_i F^{i\bar i} \bpartial_i\p_i (|\nabla u|^2) - \varphi' \p_t (|\nabla u|^2)\\
	\geq&\, 
	- 2 \varphi' \sum_{i,j} F^{i\bar i} \mathfrak{Re} \{\chi_{i\bar i j} u_{\bar j}\} + \varphi' \sum_{i,j,p} F^{i\bar i} R_{i\bar ij\bar p} u_p u_{\bar j} + \frac{\varphi'}{2} \sum_i F^{i\bar i} X^2_{i\bar i} \\
	&\, - 2 \varphi' \sum_i F^{i\bar i} \chi^2_{i\bar i} + 2 \varphi' \sum_{j} \mathfrak{Re} \{\p_j (\log\Psi) u_{\bar j}\},
\end{aligned}
\end{equation}
where
\begin{equation}
\label{second-proof-control-3}
\begin{aligned}
	& - 2 \varphi' \sum_{i,j} F^{i\bar i} \mathfrak{Re} \{\chi_{i\bar i j} u_{\bar j}\} + \varphi' \sum_{i,j,p} F^{i\bar i} R_{i\bar ij\bar p} u_p u_{\bar j}  + 2 \varphi' \sum_{j} \mathfrak{Re} \{\p_j (\log\Psi) u_{\bar j}\} \\
	&\geq - \frac{\sup_p (\sum_j |\chi_{p\bar pj}| )|\nabla u|}{K} \sum_{i} F^{i\bar i}  - \frac{\sup_{j,p,q} |R_{q\bar qj\bar p}| |\nabla u|^2}{2 K} \sum_{i} F^{i\bar i} \\
	&\qquad - \frac{|\nabla \Psi||\nabla u|}{K} .
\end{aligned}
\end{equation}
Combining  \eqref{second-proof-control-1}-\eqref{second-proof-control-3},
\begin{equation}
\label{second-proof-control-6}
\begin{aligned}
	0 \geq&\, - C_1 - C_2 \sum_i F^{i\bar i} + \frac{1}{X_{1\bar 1}} \sum_{i \neq j} \Big(\frac{S_{k - 2;ij}}{S_k} - \frac{S_{l - 2;ij}}{S_l} \Big)X_{i\bar j1} X_{j\bar i\bar 1} \\
	&\,  - \frac{1}{X^2_{1\bar 1}} \sum_i F^{i\bar i} |X_{1\bar 1i}|^2 + \varphi'' \sum_i F^{i\bar i} |\p_i (|\nabla u|^2)|^2  + \frac{\varphi'}{2} \sum_i F^{i\bar i} X^2_{i\bar i} \\
	&\,   + \rho'' \sum_i F^{i\bar i} |u_i - \ul u_i|^2 + \rho' \sum_i F^{i\bar i} (u_{i\bar i} - \ul u_{i\bar i}) - \rho' \p_t u.
\end{aligned}
\end{equation}

Now we can define 
\begin{equation}
\delta = \frac{1}{1 + 2A} = \frac{1}{1 + 6L (C_0 + 1)}
\end{equation} 
and
\begin{equation}
C_0 = \frac{C_1 + C_2 + 1}{\theta} + \frac{C_2 + 1}{\lambda} .
\end{equation}

{\bf Case 1.} $X_{n\bar n} < -\delta X_{1\bar 1}$. In this case, $X^2_{1\bar 1} \leq \frac{1}{\delta^2} X^2_{n\bar n}$ and we just need to bound $X^2_{n\bar n}$.

By \eqref{second-proof-derivative-1} and \eqref{second-test-function-varphi-1}, we have
\begin{equation}
\label{second-proof-control-8}
\begin{aligned}
	&\,\frac{1}{X^2_{1\bar 1}} \sum_i F^{i\bar i} |X_{1\bar 1i}|^2 \\
	=&\, \sum_i F^{i\bar i}  |\varphi' \p_i (|\nabla u|^2) + \rho' (u_i - \ul u_i)|^2 \\
	\leq&\, 2 (\varphi')^2 \sum_i F^{i\bar i}  | \p_i (|\nabla u|^2) |^2 +  2 (\rho')^2 \sum_i F^{i\bar i} |u_i - \ul u_i|^2 \\
	\leq&\, 
	\varphi'' \sum_i F^{i\bar i}  | \p_i (|\nabla u|^2) |^2  +  36 (C_0 + 1)^2 K \sum_i F^{i\bar i}.
\end{aligned}
\end{equation}
Substituting \eqref{second-proof-control-8} into \eqref{second-proof-control-6},
\begin{equation}
\label{second-proof-control-9}
\begin{aligned}
	&\,  C_1 + C_2 \sum_i F^{i\bar i}   
	+ 36 (C_0 + 1)^2 K \sum_i F^{i\bar i} \\
	\geq&\,   \frac{1}{8 K} \sum_i F^{i\bar i} X^2_{i\bar i} 
	+ \rho' \sum_i F^{i\bar i} (u_{i\bar i} - \ul u_{i\bar i}) - \rho' \p_t u.
\end{aligned}
\end{equation}

According to Lemma~\ref{quotient-alternate-corollary}, there are at most two possibilities. If \eqref{quotient-alternate-corollary-1} holds true,
\begin{equation}
\label{second-proof-control-10}
\begin{aligned}
	&\, C_1 + C_2 \sum_i F^{i\bar i} + 36 (C_0 + 1)^2 K \sum_i F^{i\bar i} \\
	\geq&\,  \frac{X^2_{n \bar n}}{8 n K}  \sum_i F^{i\bar i} + \theta (C_0 + 1) + \theta (C_0 + 1) \sum_i F^{i\bar i} .
\end{aligned}
\end{equation}
Then
\begin{equation}
\label{second-proof-bound-1}
	X_{1\bar 1} < \frac{ \sqrt{288 n} (C_0 + 1)}{\delta} K .
\end{equation}
If \eqref{quotient-alternate-corollary-2} holds true,
\begin{equation}
\label{second-proof-control-11}
\begin{aligned}
	&\, C_1 + C_2 \sum_i F^{i\bar i} 
	+ 36 (C_0 + 1)^2 K \sum_i F^{i\bar i} \\
	\geq&\,  \frac{\theta X_{1\bar 1}}{8 K}   + \frac{X^2_{n \bar n}}{8 n K}  \sum_i F^{i\bar i}  + \rho' \sum_i F^{i\bar i} (u_{i\bar i} - \ul u_{i\bar i}) - \rho' u_t.
\end{aligned}
\end{equation}
Since
\begin{equation}
\label{second-proof-control-7}
\begin{aligned}
	&\, \rho' \sum_i F^{i\bar i} (u_{i\bar i} - \ul u_{i\bar i}) - \rho' u_t \\
	\geq&\, \rho' \log \frac{S_k}{S_l} - \rho' u_t - \lambda \rho' \sum_i  F^{i\bar i} \\
	\geq&\, - 3 (C_0 + 1) \sup_M |\log \Psi| + \lambda (C_0 + 1) \sum_i  F^{i\bar i}  ,
\end{aligned}
\end{equation}
it follows that
\begin{equation}
\label{second-proof-control-12}
	 C_1 + 3 (C_0 + 1) \sup_M |\log \Psi|
	+ 36 (C_0 + 1)^2 K \sum_i F^{i\bar i} 
	\geq \frac{\theta  X_{1\bar 1}}{8 K}  + \frac{X^2_{n \bar n}}{8 n K}  \sum_i F^{i\bar i}  .
\end{equation}
Then we either have \eqref{second-proof-bound-1} or
\begin{equation}
\label{second-proof-bound-2}
	X_{1\bar 1} \leq \frac{8 (C_1 + 3(C_0 + 1) \sup_M |\log \Psi|)}{\theta} K .
\end{equation}

{\bf Case 2.} $X_{n\bar n} \geq -\delta X_{1\bar 1}$. Define 
\begin{equation}
I = \left\{i \in \{1 , \cdots, n\} \,\Big|\,F^{i\bar i} > \delta^{- 1} F^{1\bar 1} \right\}.
\end{equation}
Then
\begin{equation}
\label{second-proof-control-13}
\begin{aligned}
	&\, \frac{1}{X_{1\bar 1}} \sum_{i \neq j} \Big(\frac{S_{k - 2;ij}}{S_k} - \frac{S_{l - 2;ij}}{S_l} \Big)X_{i\bar j1} X_{j\bar i\bar 1} \\
	\geq&\, \frac{1 - \delta}{1 + \delta} \frac{1}{X^2_{1\bar 1}} \sum_{i \in I} F^{i\bar i} \big(|X_{1\bar 1i}|^2 + 2 \mathfrak{Re}\{X_{1\bar 1i} \bar b_i\}\big),
\end{aligned}
\end{equation}
where $b_i = \chi_{i\bar 1 1} - \chi_{1\bar 1i}$ . So we have
\begin{equation}
\label{second-proof-control-14}
\begin{aligned}
	&\, C_1 + C_2 \sum_i F^{i\bar i} + \frac{1}{X^2_{1\bar 1}} \sum_i F^{i\bar i}  |X_{1\bar 1i}|^2  \\
	\geq &\, \frac{1 - \delta}{1 + \delta} \frac{1}{X^2_{1\bar 1}} \sum_{i \in I} F^{i\bar i} \big(|X_{1\bar 1i}|^2 + 2 \mathfrak{Re}\{X_{1\bar 1i} \bar b_i\}\big) \\
	&\, + \varphi'' \sum_i F^{i\bar i} |\p_i (|\nabla u|^2)|^2  + \rho'' \sum_{i \in I} F^{i\bar i} |u_i - \ul u_i|^2  \\
	&\,   + \frac{1}{8 K} \sum_i F^{i\bar i} X^2_{i\bar i} + \rho' \sum_i F^{i\bar i} (u_{i\bar i} - \ul u_{i\bar i}) - \rho' u_t.
\end{aligned}
\end{equation}
We need to control the terms in \eqref{second-proof-control-14}. By \eqref{second-proof-derivative-1} and the fact that $\varphi'' = 2 (\varphi')^2$,
\begin{equation}
\label{second-proof-control-15}
\begin{aligned}
	\varphi'' \sum_{i \in I} F^{i\bar i} |\p_i (|\nabla u|^2)|^2 
	\geq 2 \sum_{i \in I} F^{i\bar i} \Big(\delta \Big|\frac{X_{1\bar 1i}}{X_{1\bar 1}}\Big|^2 - \frac{\delta}{1 - \delta} |\rho' (u_i - \ul u_i)|^2\Big) ,
\end{aligned}
\end{equation}
and in addition using the fact that $\rho'' \geq \frac{2 \delta}{1 - \delta} (\rho')^2$ and Schwarz inequality,
\begin{equation}
\label{second-proof-control-16}
\begin{aligned}
	&\,  \frac{1 - \delta}{1 + \delta} \frac{1}{X^2_{1\bar 1}} \sum_{i \in I} F^{i\bar i} \Big(|X_{1\bar 1i}|^2 + 2 \mathfrak{Re}\{X_{1\bar 1i} \bar b_i\}\Big)  - \frac{1}{X^2_{1\bar 1}} \sum_{i \in I} F^{i\bar i} |X_{1\bar 1i}|^2 \\
	&\,   + \rho'' \sum_{i \in I} F^{i\bar i} |u_i - \ul u_i|^2  + \varphi'' \sum_{i \in I} F^{i\bar i} |\p_i (|\nabla u|^2)|^2 \\
	\geq 
	&\,\frac{1 - \delta}{1 + \delta} \frac{1}{X^2_{1\bar 1}} \sum_{i \in I} F^{i\bar i} \Big(|X_{1\bar 1i}|^2 + 2 \mathfrak{Re}\{X_{1\bar 1i} \bar b_i\}\Big) - \frac{1}{X^2_{1\bar 1}} \sum_{i \in I} F^{i\bar i}  |X_{1\bar 1i}|^2  \\
	&\,   + \rho'' \sum_{i \in I} F^{i\bar i} |u_i - \ul u_i|^2 + 2 \sum_{i \in I} F^{i\bar i} \Big(\delta \Big|\frac{X_{1\bar 1i}}{X_{1\bar 1}}\Big|^2 - \frac{\delta}{1 - \delta} |\rho' (u_i - \ul u_i)|^2\Big) \\
	\geq&\, \frac{2 \delta^2}{1 + \delta} \frac{1}{X^2_{1\bar 1}} \sum_{i \in I} F^{i\bar i} |X_{1\bar 1i}|^2 + \frac{2 (1 - \delta)}{1 + \delta} \frac{1}{X^2_{1\bar 1}} \sum_{i \in I} F^{i\bar i}  \mathfrak{Re}\{X_{1\bar 1i} \bar b_i\} \\
	\geq&\, \frac{\delta^2}{X^2_{1\bar 1}} \sum_{i \in I} F^{i\bar i} |X_{1\bar 1i}|^2 - \frac{1 - \delta}{(1 + \delta) \delta^2} \frac{\sum_p  |b_p|^2}{X^2_{1\bar 1}} \sum_{i\in I} F^{i\bar i}.
\end{aligned}
\end{equation}
For the terms without index in $I$, by \eqref{second-proof-control-8}
\begin{equation}
\label{second-proof-control-17}
\begin{aligned}
&\,	\varphi'' \sum_{i\notin I}	 F^{i\bar i} |\p_i (|\nabla u|^2)|^2  - \frac{1}{X^2_{1\bar 1}} \sum_{i\notin I} F^{i\bar i} |X_{1\bar 1i}|^2 \\
	\geq&\, - 36 (C_0 + 1)^2 K \max_{i \notin I}  F^{i\bar i} \\
	\geq&\, - \frac{36 (C_0 + 1)^2 K}{\delta} F^{1\bar 1} .
\end{aligned}
\end{equation}
We may assume that
\begin{equation}
	X^2_{1\bar 1} \geq \frac{ (1 - \delta)}{ ( 1+ \delta) \delta^2} \sum_p |b_p|^2,
\end{equation}
otherwise, the $C^2$ bound is achieved.
Substituting \eqref{second-proof-control-16} and \eqref{second-proof-control-17} into \eqref{second-proof-control-14},
\begin{equation}
\label{second-proof-control-18}
\begin{aligned}
	&\,   C_1 + (C_2 + 1) \sum_i F^{i\bar i} + \frac{36 (C_0 + 1)^2 K}{\delta} F^{1\bar 1} \\
	\geq&\, \frac{1}{8 K} \sum_i F^{i\bar i} X^2_{i\bar i}   + \rho' \sum_i F^{i\bar i} (u_{i\bar i} - \ul u_{i\bar i}) - \rho' u_t .
\end{aligned}
\end{equation}
If \eqref{quotient-alternate-corollary-1} holds true,
\begin{equation}
\label{second-proof-control-19}
\begin{aligned}
	&\,   C_1 + (C_2 + 1) \sum_i F^{i\bar i} + \frac{36 (C_0 + 1)^2 K}{\delta} F^{1\bar 1} \\
	\geq&\, \frac{1}{8 K} \sum_i F^{i\bar i} X^2_{i\bar i}   + \theta (C_0 + 1) + \theta (C_0 + 1) \sum_i F^{i\bar i}.
\end{aligned}
\end{equation}
Then
\begin{equation}
\label{second-proof-control-20}
	\frac{36 (C_0 + 1)^2 K}{\delta} F^{1\bar 1} \geq \frac{1}{8 K} \sum_i F^{i\bar i} X^2_{i\bar i} .
\end{equation}
So we have
\begin{equation}
\label{second-proof-bound-3}
	X_{1\bar 1} \leq \sqrt{\frac{288}{\delta}} (C_0 + 1) K .
\end{equation}
If \eqref{quotient-alternate-corollary-2} holds true,
\begin{equation}
\label{second-proof-control-21}
\begin{aligned}
	&\,   C_1 + (C_2 + 1) \sum_i F^{i\bar i}  + \frac{36 (C_0 + 1)^2 K}{\delta} F^{1\bar 1}\\
	\geq&\, \frac{1}{8 K} \sum_i F^{i\bar i} X^2_{i\bar i}   - 3 (C_0 + 1) \sup_M |\log \Psi| + \lambda (C_0 + 1) \sum_i  F^{i\bar i} .
\end{aligned}
\end{equation}
Then
\begin{equation}
\label{second-proof-control-22}
	 C_1 + 3 (C_0 + 1) \sup_M |\log \Psi| + \frac{36 (C_0 + 1)^2 K}{\delta} F^{1\bar 1}
	\geq \frac{\theta X_{1\bar 1}}{16 K} + \frac{1}{16 K} F^{i\bar i} X^2_{1\bar 1}    .
\end{equation}
So we have either
\begin{equation}
\label{second-proof-bound-4}
	X_{1\bar 1} \leq \frac{16(C_1 +  3 (C_0 + 1) \sup_M |\log \Psi|)}{\theta} K
\end{equation}
or
\begin{equation}
\label{second-proof-bound-5}
	X_{1\bar 1} \leq \frac{24 (C_0 + 1)}{\sqrt{\delta}} K.
\end{equation}

\end{proof}

\section{The gradient estimate}
\label{gradient}

In this section, we shall adapt the blow up argument of Dinew and Kolodziej\cite{DK} and Gill\cite{Gill14}.

\begin{theorem}
\label{gradient-theorem-1}
On the maximal time interval $[0,T)$, there is a  uniform constant $C > 0$ such that
\begin{equation}
\label{gradient-theorem-1-inequality}
	\sup_{M \times [0,T)} |\nabla u| \leq C .
\end{equation}
\end{theorem}
\begin{proof}

The argument is very similar to those of Dinew and Kolodziej\cite{DK} and Gill\cite{Gill14}, so we just give a brief statement here.

We shall prove the theorem by contradiction and suppose that the gradient estimate~\eqref{gradient-theorem-1-inequality} does not hold. Then there exists a sequence $(x_m, t_m) \in M \times [0,T)$ with $t_m \rightarrow T$ such that $\lim_{m\rightarrow \infty} |\nabla u (x_m, t_m)| \rightarrow \infty$ and $|\nabla u(x_m,t_m)| = \sup_{M\times[0,t_m]} |\nabla u|$. We set $C_m := |\nabla u(x_m,t_m)| $.

After passing to a subsequence, we may assume that $x_m \rightarrow x \in M$. Fix a normal coordinate chart around $x$, which we identify with an open set in $\mathbb{C}^n$ with coordinates $(z^1 , \cdots , z^n)$, and such that $\omega (0) = \beta := \sum_{i,j} \delta_{ij} dz^i \wedge d\bar z^j$. Without loss of generality, we may assume that the open set contains $\ol {B_1 (0)}$. We define, on the ball $\ol{B_{C_m} (0)}$ in $\mathbb{C}^n$,
\begin{equation}
	\tilde u_m (z) : = u_m \Big(\frac{z}{C_m}\Big) .
\end{equation}
By passing to a subsequence again, we can find a limit function $\tilde u \in C^{1,\alpha} (\mathbb{C}^n)$. As show in \cite{DK}, it is sufficient to prove that $\tilde u$ is a maximal $k-sh$ function.
Without loss of generality, we may assume that $\tilde u_m$ is $C^{1,\alpha}$ convergent to $\tilde u$. Then we have
\begin{equation}
\begin{aligned}
	&\,\left[\chi_u \Big(\frac{z}{C_m}\Big) \right]^k \wedge \left[\omega \Big(\frac{z}{C_m}\Big) \right]^{n - k} \\
	=&\, e^{ \p_t u}\psi_m  \Big(\frac{z}{C_m}\Big) \left[\chi_u \Big(\frac{z}{C_m}\Big)  \right]^l \wedge \left[\omega \Big(\frac{z}{C_m}\Big) \right]^{n - l} .
\end{aligned}
\end{equation}
Fixing $z$, we have
\begin{equation}
\begin{aligned}
	&C^{2(k - l)}_m \left[ O\Big(\frac{1}{C^2_m}\Big) \beta + \frac{\sqrt{- 1}}{2} \p \bpartial\tilde  u_m (z) \right]^k \wedge \left[\left(1 + O \Big(\frac{|z|^2}{C^2_m}\Big)\right) \beta \right]^{n - k} \\
	=& e^{\p_t u}\psi_m \Big(\frac{z}{C_m}\Big) \left[ O\Big(\frac{1}{C^2_m}\Big) \beta + \frac{\sqrt{- 1}}{2} \p \bpartial\tilde  u_m (z) \right]^l \wedge \left[ \left(1 + O \Big(\frac{|z|^2}{C^2_m}\Big)\right) \beta \right]^{n - l} .
\end{aligned}
\end{equation}
Since $\p_t u$ is bounded, 
\begin{equation}
\label{gradient-eq-6}
	\left[\frac{\sqrt{-1}}{2}\p \bpartial \tilde u (z) \right]^k \wedge \beta^{n - k} = 0 ,
\end{equation}
which is in the pluripotential sense.
Moreover, a similar reasoning tells us that for any $1\leq p \leq k$,
\begin{equation}
\label{gradient-eq-7}
	\left[\frac{\sqrt{-1}}{2}\p \bpartial \tilde u (z) \right]^p \wedge \beta^{n - p} \geq 0 .
\end{equation}
By a result of Blocki\cite{Blocki05}, the above \eqref{gradient-eq-6} and \eqref{gradient-eq-7} imply that $\tilde u$ is a maximal $k - sh$ function in $\mathbb{C}^n$. 

\end{proof}

\section{Long time existence and Convergence}
\label{convergence}

If $T > 0$ is a real number, Theorem~\ref{gradient-theorem-1} implies that there is a time-independent gradient estimate on $[0,T)$. By the Evans-Krylov theorem and Schauder estimates, we can obtain $C^\infty$ estimates on $[0,T)$. Then standard procedure based on implicit function theorem can extend $u(x,t)$ to $[0,T+\epsilon)$ for some small $\epsilon > 0$, which contradicts the definition of $T$. So, $T$ must be $\infty$.

Applying Theorem~\ref{gradient-theorem-1}, the Evans-Krylov theorem and Schauder estimates again, we obtain $C^\infty$ estimates on $[0,\infty)$. 
Now we are able to show the convergence of the solution flow.
The arguments of Gill\cite{Gill11} following Cao\cite{Cao85} can be applied verbatim here, and thus the proof is omitted. 

From the arguments, 
\begin{equation}
	\sup_{x\in M} \frac{\p u}{\p t} (x,t) - \inf_{x\in M} \frac{\p u}{\p t} (x,t) \leq C e^{- c_0 t} ,
\end{equation}
for some $c_0 > 0$. Noticing that 
\begin{equation}
	\int_M \frac{\p \hat u}{\p t} \chi^l_u \wedge \omega^{n - l} = 0,
\end{equation}
for any fixed $t$ there must be $y \in M$ such that $\p_t \hat u (y,t) = 0$. Therefore
\begin{equation}
\label{convergence:solution-normalization-decay}
\begin{aligned}
	\left|\frac{ \p \hat u (x,t)}{\p t}\right| &= \left|\frac{ \p \hat u (x,t)}{\p t} - \frac{ \p \hat u (y,t)}{\p t}\right| \\
	&\leq \sup_{x\in M} \frac{\p u}{\p t} (x,t) - \inf_{x\in M} \frac{\p u}{\p t} (x,t) \leq C e^{- c_0 t}  ,
\end{aligned}
\end{equation}
and thus
\begin{equation}
\label{convergence:solution-normalization-decay-1}
	\frac{\p }{\p t} \left(\hat u + \frac{C}{c_0} e^{- c_0 t}\right) \leq 0 .
\end{equation}
By the $L^\infty$ estimate, it is easy to see that $\hat u + \frac{C}{c_0} e^{- c_0 t}$ is bounded.
By a standard argument, \eqref{convergence:solution-normalization-decay-1} implies that $\hat u$ is $C^\infty$ convergent to a smooth function $\hat u_\infty$.

Rewriting equation~\eqref{pfcqe-int-flow-equation},
\begin{equation}
\label{convergence:solution-normalization-flow}
	\frac{\p \hat u}{\p t} + \frac{\p}{\p t} \frac{J_l (u)}{\int_M \chi^l \wedge \omega^{n - l}} = \log \frac{\chi^k_{\hat u} \wedge \omega^{n - k}}{ \chi^l_{\hat u} \wedge \omega^{n - l} } - \log \psi.
\end{equation}
Letting $t\rightarrow\infty$,  \eqref{convergence:solution-normalization-decay} implies that \eqref{convergence:solution-normalization-flow} converges to a constant $b$.

\noindent
{\bf Acknowledgements}\quad
The author is very grateful to Bo Guan for his encouragement and helpful conversations. 



\end{document}